\pdfoutput=1
\RequirePackage{ifpdf}
\ifpdf 
\documentclass[pdftex]{sigma}
\else
\documentclass{sigma}
\fi

\usepackage{mathrsfs}
\usepackage{tikz-cd}

\DeclareMathOperator\Ran{Ran}

\newcommand{\R}{{\mathbb R}}
\newcommand{\C}{{\mathbb C}}
\newcommand{\N}{{\mathbb N}}
\newcommand{\Q}{{\mathbb Q}}
\newcommand{\Z}{{\mathbb Z}}

\newcommand{\jcal}{{\mathcal J}}
\newcommand{\kcal}{{\mathcal K}}

\newcommand{\pcal}{{\mathcal P}}
\newcommand{\fancyI}{\mathcal{I}}
\newcommand{\fancyJ}{\mathcal{J}}
\newcommand{\fancyL}{\mathcal{L}}

\numberwithin{equation}{section}

\newtheorem{Theorem}{Theorem}[section]
\newtheorem*{Theorem*}{Theorem}

\newtheorem{Proposition}[Theorem]{Proposition}
 { \theoremstyle{definition}

 }

\begin{document}


\newcommand{\arXivNumber}{2501.07756}

\renewcommand{\PaperNumber}{035}

\FirstPageHeading

\ShortArticleName{A Pile of Shifts II: Structure and $K$-Theory}

\ArticleName{A Pile of Shifts II: Structure and $\boldsymbol{K}$-Theory}

\Author{Shelley HEBERT~$^{\rm a}$, Slawomir KLIMEK~$^{\rm b}$, Matt MCBRIDE~$^{\rm a}$ and J. Wilson PEOPLES~$^{\rm c}$}

\AuthorNameForHeading{S.~Hebert, S.~Klimek, M.~McBride and J.W.~Peoples}

\Address{$^{\rm a)}$~Department of Mathematics and Statistics, Mississippi State University,\\
\hphantom{$^{\rm a)}$}~175 President's Cir. Mississippi State, MS 39762, USA}
\EmailD{\href{mailto:sdh7@msstate.edu}{sdh7@msstate.edu}, \href{mailto:mmcbride@math.msstate.edu}{mmcbride@math.msstate.edu}}

\Address{$^{\rm b)}$~Department of Mathematical Sciences, Indiana University Indianapolis,\\
\hphantom{$^{\rm b)}$}~402 N. Blackford St., Indianapolis, IN 46202, USA}
\EmailD{\href{mailto:klimek@iu.edu}{klimek@iu.edu}}

\Address{$^{\rm c)}$~Department of Mathematics, Pennsylvania State University, \\
\hphantom{$^{\rm c)}$}~107 McAllister Bld., University Park, State College, PA 16802, USA}
\EmailD{\href{mailto:peoplesjwilson@gmail.com}{peoplesjwilson@gmail.com}}

\ArticleDates{Received January 16, 2025, in final form May 06, 2025; Published online May 12, 2025}

\Abstract{We discuss $C^*$-algebras associated with several different natural shifts on the Hilbert space of the $s$-adic tree, continuing the analysis from [\textit{Banach~J. Math. Anal.} \textbf{19} (2025), 32, 30~pages, arXiv:2412.00854] and in particular we describe their structure and compute the $K$-theory groups.}


\Classification{47L65; 46L35; 46L80}

\section{Introduction}
Let $\Z_s$ be the ring of $s$-adic integers, where $s$ is a positive integer, not necessarily a prime. Let~$\mathcal{V}$ be the set of balls in $\Z_s$. The corresponding Hilbert space $\ell^2(\mathcal{V})$ carries a natural, diagonal representation of the algebra of continuous functions on $\Z_s$. The purpose of this paper and the companion paper \cite{HKMP3} is to study natural $C^*$-algebras in $\ell^2(\mathcal{V})$ generated by $C(\Z_s)$ along with specific shifts. Those shift algebras are interesting examples of $C^*$-algebras with connections to number theory and combinatorics, a subject that has gained some popularity, see for example~\cite{CL,R}.

Our motivation to study $C^*$-algebras associated with shifts on trees partially comes from trying to extend the results from \cite{KMR, KRS1, KRS2} on $p$-adic spectral triples as well as ideas from~\cite{PB} on noncommutative geometry of ultrametric Cantor sets. A heuristic goal is to understand ``quantum'' non-archimedean spaces and their geometry with seemingly many candidates for such concepts.

We review all the definitions and the key results from the companion paper \cite{HKMP3} so that the present paper is self-contained. The main purpose of \cite{HKMP3} was to determine how each of those shift algebras can be identified with a crossed product by an endomorphism, in fact a~monomorphism with a hereditary range. In this paper, we delve deeper into the structure of the shift algebras and compute their $K$-theory.

There are several competing concepts of crossed product by an endomorphism \cite{ABL,BKR,E,Mu,St} but they all agree for monomorphisms with a hereditary range. In \cite{HKMP3}, we took a pragmatic approach to crossed products by finding enlargements of the initial algebra $C(\Z_s)$ up to a~better behaved ``coefficient'' $C^*$-subalgebras of the shift algebras which makes the corresponding monomorphisms to have hereditary range. Those coefficient algebras are the starting point of the current paper.

To understand the structure of the shift algebras, we consider natural ideals in each of the algebras and then describe their quotients. The main technique is an analog of a basic construction for usual crossed products by automorphisms involving invariant ideals, see \cite[Section~3.4]{Williams}. In particular, we prove a counterpart to \cite[Proposition 3.19]{Williams} which gives a short exact sequence of crossed products obtained from an invariant ideal.

Another key technique used in this paper to understand the structure of some crossed products is to make them concrete algebras through faithful representations. Faithfulness is established using a version of the O'Donovan condition easily implied by the existence of gauge actions.

A short exact sequence of shift $C^*$-algebras induces a $6$-term exact sequence in $K$-theory which we then, in each case, analyze and reduce to short exact sequences. This is enough to determine the K-groups in all but the Bunce--Deddens shift case for which the sequence does not split. Instead, in this case, we describe the $K$-theory as a direct limit of finitely generated free Abelian groups.

The paper is organized as follows. Preliminary Section \ref{sec2} contains the definitions of shift algebras and a gauge action for each. Crossed products by endomorphisms, coefficient algebras with gauge actions and invariant ideals are the subject of Section \ref{sec3}. Each of the following four sections is devoted to one of the shift algebras. Those sections are structured similarly: we recall the coefficients algebras from \cite{HKMP3}, describe the coefficient algebra ideals, their quotients and the quotient endomorphisms, followed by the identification of the corresponding shift algebra ideals and their quotients. Each section is then concluded with a calculation of the $K$-theory.

\section{Preliminaries}\label{sec2}
\subsection[s-adic integers]{$\boldsymbol{s}$-adic integers}
For $s\in\N$ with $s\ge2$, the space of $s$-adic integers $\Z_s$ is defined as consisting of infinite sums
\begin{equation*}
\Z_s= \Biggl\{x=\sum_{j=0}^\infty x_j s^j\mid 0\le x_j\le s-1, \, j=0,1,2,\dots\Biggr\} .
\end{equation*}

We call the above expansion of $x$ its $s$-adic expansion. Note that $\Z_s$ is isomorphic with the countable product $\prod \Z/s\Z$ and so, when the product is equipped with the Tychonoff topology, it is a Cantor set. It is also a metric space with the usual norm $|\cdot|_s$, which is defined as follows
$
|0|_s=0$ and, if $x\ne 0$, $ |x|_s=s^{-n}$,
where $x_n$ is the first nonzero term in the above $s$-adic expansion.
Moreover, $\Z_s$ is an Abelian ring with unity with respect to addition and multiplication with a carry and we have
\smash{$\Z_s\cong\lim\limits_{\longleftarrow}\Z/s^n\Z$}.
Using the division algorithm modulo~$s$ repeatedly, we can associate an $s$-adic expansion to any integer and so $\Z_s$ naturally contains a copy of $\Z$ which is dense in $\Z_s$.

Any ball in $\Z_s$ is of the form
\begin{equation*}
\left\{y\in\Z_s\mid |y-x|_s\le\frac{1}{s^n}\right\},
\end{equation*}
where $n\in\Z_{\ge0}$ and $0\le x<s^n$. Here $x$ is a particular center of this ball. Consequently, we can write the set of balls $\mathcal{V}$ in the following way
$
\mathcal{V}=\{(n,x)\mid n=0,1,2,\dots,\, 0\le x<s^n\} $.

\subsection{Hilbert spaces and shifts}

Throughout this paper, we use the Hilbert space $H=\ell^2(\mathcal{V})$ with canonical basis $\{E_{(n,x)}\}$, where~${n=0,1,2,\dots}$ and $0\le x<s^n$. Let $f\in C(\Z_s)$ and define a bounded operator $M_f\colon H\to H$ by
$M_fE_{(n,x)}=f(x)E_{(n,x)}$.
Due to the density of $\Z_{\ge0}$ in $\Z_s$, we have that
\begin{equation*}
\|M_f\| = \sup_{x\in\Z_{\ge0}}{|f(x)|} = \sup_{x\in\Z_s}{|f(x)|}
\end{equation*}
and so $f\mapsto M_f$ is a faithful representation of the $C^*$-algebra $C(\Z_s)$.

There are four main shifts in $H$ that are the focus of this paper and which we call the Bunce--Deddens shift, the Hensel shift, the Bernoulli shift, and the Serre shift. All four shifts are defined on $H$ with formulas given on the basis elements of $H$ as follows:
\begin{enumerate}\itemsep=0pt
\item[(1)] \textit{The Bunce--Deddens shift:} $U\colon H\to H$ via
$
UE_{(n,x)}= E_{(n+1, x+1)}$.
\item[(2)] \textit{The Hensel shift:} $V\colon H\to H$ via
$
VE_{(n,x)} = E_{(n+1, sx)}$.
\item[(3)] \textit{The Bernoulli shift:} $S\colon H\to H$ via
\begin{equation*}
SE_{(n,x)}=\frac{1}{s}\sum_{j=0}^{s-1}E_{(n+1, sx+j)} .
\end{equation*}
\item[(4)] \textit{The Serre shift:} $W\colon H\to H$ via
\begin{equation*}
WE_{(n,x)} = \frac{1}{\sqrt{s}}\sum_{j=0}^{s-1}E_{(n+1, x+js^n)} .
\end{equation*}
\end{enumerate}

A computation yields their respective adjoints
\begin{gather*}
(1) \quad U^*E_{(n,x)} =
\begin{cases}
E_{(n-1, x-1)} &\textrm{if}\ n\ge1,\ 0<x\le s^{n-1} ,\\
0 &\textrm{if}\ n=0 \ \textrm{or}\ x=0\ \textrm{or}\ s^{n-1}<x<s^n,
\end{cases}
\\
(2) \quad
V^*E_{(n,x)} =
\begin{cases}
E_{(n-1, x/s)} &\textrm{if}\ n\ge1,\ s\mid x,\\
0 &\textrm{if}\ n=0\ \textrm{or}\ s\nmid x,
\end{cases}\\
(3)\quad
S^*E_{(n,x)} =
\begin{cases}
E_{(n-1, (x-x\textrm{ mod }s)/s)} &\textrm{if}\ n\ge1,\\
0&\textrm{if}\ n=0,
\end{cases}\\
(4) \quad
W^*E_{(n,x)} =
\begin{cases}
\dfrac{1}{\sqrt{s}}E_{(n-1, x\textrm{ mod }s^{n-1})} &\textrm{if}\ n\ge1,\\
0&\textrm{if}\ n=0.
\end{cases}
\end{gather*}
It is easy to verify that
$U^*U=V^*V=S^*S=W^*W=I$,
hence each one is an isometry. For each of these four shifts $\mathcal{J}\in\{U, V, S, W\}$, we have the corresponding shift $C^*$-algebras
\begin{equation*}
A_{\mathcal{J}} := C^*(\mathcal{J},M_f \mid f\in C(\Z_s)) ,
\end{equation*}
which are the main objects of interest in this paper.

By analogy with graph algebras, a gauge action is an important structure shared by all those shift algebras. For $\theta\in\R/\Z$, consider the following strongly continuous one-parameter family~$\{\mathcal{U}_\theta\}$ of unitary operators on $\ell^2(\mathcal{V})$ defined on basis elements by
$
\mathcal{U}_\theta E_{(n,x)}={\rm e}^{2\pi {\rm i}n\theta}E_{(n,x)}$.
We define a continuous one-parameter family $\{\rho_\theta\}$ of automorphisms of $A_{\mathcal{J}}$, the gauge action~${\rho_\theta\colon A_{\mathcal{J}}\to A_{\mathcal{J}}}$, by the formula
$\rho_\theta(a)=\mathcal{U}_\theta a \mathcal{U}_\theta^{-1}$.
Then on generators of $A_{\mathcal{J}}$ we have
\begin{equation*}
\rho_\theta(\mathcal{J})={\rm e}^{2\pi {\rm i}\theta}\mathcal{J} ,\qquad \rho_\theta(\mathcal{J}^*)={\rm e}^{-2\pi {\rm i}\theta}\mathcal{J}^* ,\qquad\textrm{and}\qquad \rho_\theta(M_f) = M_f .
\end{equation*}

Throughout the rest of this paper, we describe the structure of each algebra and study its $K$-theory.

\section{Coefficient algebras with a gauge action}\label{sec3}
\subsection{Coefficient algebras}
The notion of a coefficient algebra was introduced in~\cite{LO}. A review of this concept, with additional material suitable for our purposes, is the main theme of this section.

Given a Hilbert space $\mathcal{H}$, let ${\mathcal{J}} \in B(\mathcal{H})$ be an isometry and let $B_\jcal \subseteq B(\mathcal{H})$ be a $C^*$-subalgebra of $B(\mathcal{H})$ containing the identity $I$. We use the notation $A_{\mathcal{J}}$ for the $C^*$-algebra generated by $B_\jcal$ and ${\mathcal{J}}$. We call $B_\jcal$ a {\it coefficient algebra} of the $C^*$-algebra $A_{\mathcal{J}}$ if $B_\jcal$ and ${\mathcal{J}}$ satisfy the following conditions
$
\alpha_\jcal(a):={\mathcal{J}} a{\mathcal{J}}^* \in B_\jcal
$
and
$
\beta_\jcal(a):={\mathcal{J}}^*a{\mathcal{J}} \in B_\jcal$,
for all $a\in B_\jcal$.

It is easy to see that $\alpha_\jcal$ is an endomorphism of $B_\jcal$ and we have the following properties~${
\beta_\jcal\alpha_\jcal(a)=a}$, $ \alpha_\jcal\beta_\jcal(a)=\alpha_\jcal(I)a\alpha_\jcal(I)$, and $ \beta_\jcal(I)=I$.
Moreover, the map $\beta_\jcal\colon B_\jcal\to B_\jcal$ is linear, continuous, positive and has the ``transfer'' property~${
\beta_\jcal(\alpha_\jcal(a)b)=a\beta_\jcal(b)}$,
and thus is an example of Exel's transfer operator \cite{E}. Also notice that if the algebra $B_\jcal$ is commutative, or more generally if $\alpha_\jcal(I)$ is in the center of $B_\jcal$, then~$\beta_\jcal$ is also an endomorphism. Indeed, notice that $\alpha_\jcal(I) = \jcal\jcal^*$ and we have that
\begin{equation*}
\beta_\jcal(a)\beta_\jcal(b) = \jcal^*a\jcal\jcal^*b\jcal=\jcal^*a\alpha_\jcal(I)b\jcal=\jcal^*\alpha_\jcal(I)ab\jcal=\jcal^*\jcal\jcal^*ab\jcal=\beta_\jcal(ab) .
\end{equation*}

Clearly, $\alpha_\jcal(I)$ is a projection and notice that $\textrm{Ran}(\alpha_\jcal)$, the range of $\alpha_\jcal$, is a hereditary subalgebra of $B_\jcal$ and we have
$\Ran(\alpha_\jcal)=\alpha_\jcal(I)B_\jcal\alpha_\jcal(I)$.

\subsection{Gauge action}
We now consider additional structure called a gauge action. Assume that there is a one-parameter family $\{\mathcal{U}_\theta\}$, for $\theta\in\R/\Z$, of unitary operators on $\mathcal{H}$ satisfying the following commutation properties
$\mathcal{U}_\theta a = a \mathcal{U}_\theta $, $ \mathcal{U}_\theta \mathcal{J} = {\rm e}^{2\pi {\rm i}\theta}\mathcal{J} \mathcal{U}_\theta
$
for all $a\in B_\jcal$.
We define a one-parameter family~$\{\rho_\theta\}$ of automorphisms $\rho_\theta\colon B(\mathcal{H})\to B(\mathcal{H})$ by the formula
$
\rho_\theta(a)=\mathcal{U}_\theta a \mathcal{U}_\theta^{-1}$.
It follows that on generators of $A_{\mathcal{J}}$ we have
\[
\rho_\theta(\mathcal{J})={\rm e}^{2\pi {\rm i}\theta}\mathcal{J}, \qquad\rho_\theta(\mathcal{J}^*)={\rm e}^{-2\pi {\rm i}\theta}\mathcal{J}^*,\qquad \text{and} \qquad\rho_\theta(a) = a,\quad a\in B_\jcal.
\]
Clearly, the map $\theta\mapsto \rho_\theta(x)$ is continuous for all $x$ which are polynomials in $\mathcal{J}$ and $\mathcal{J}^*$ with coefficients in $B_\jcal$. It follows by an approximation argument that it is continuous for every~${x\in A_{\mathcal{J}}}$.

As a consequence of the existence of a gauge action, we have an expectation
$E\colon A_\mathcal{J}\to A_\mathcal{J}$
given by the formula
\begin{equation}\label{expect_formula}
E(x):=\int_0^1 \rho_\theta(x) {\rm d}\theta.
\end{equation}

The authors of \cite{HKMP3} noted that $C(\Z_s)$ is not a coefficient algebra for the shift algebras currently considered. Instead, coefficient algebra extensions of $C(\Z_s)$ were constructed as the invariant subalgebras $A_{\mathcal{J},\textrm{inv}}$ of $A_\mathcal{J}$ with respect to the automorphism $\rho_\theta$.
This is because $C(\Z_s)\subseteq A_{\mathcal{J},\textrm{inv}}$ and, for every $a\in A_{\mathcal{J},\textrm{inv}}$, we also have
$
\mathcal{J}^*a\mathcal{J}\in A_{\mathcal{J},\textrm{inv}}$ and $ \mathcal{J}a\mathcal{J}^*\in A_{\mathcal{J},\textrm{inv}}$.

\subsection{O'Donovan condition}
Using the commutation relation
${\mathcal{J}} a=\alpha_\jcal(a){\mathcal{J}}$,
one can arrange the polynomials in ${\mathcal{J}}$, ${\mathcal{J}}^*$ with coefficients in $B_\jcal$ so that the coefficients are in front of the powers of ${\mathcal{J}}$ and behind the powers of ${\mathcal{J}}^*$. This leads to the following observation (see~\cite[Proposition~2.3]{LO}):
the vector space~$\mathcal{F}(B_\jcal,{\mathcal{J}})$ consisting of finite sums
\begin{equation}\label{FourierV}
x=\sum_{n\ge0}a_n {\mathcal{J}}^n + \sum_{n<0}({\mathcal{J}}^*)^{-n} a_n ,
\end{equation}
where $a_n\in B_\jcal$, is a dense $*$-subalgebra of $A_{\mathcal{J}}$.
Additionally, using ${\mathcal{J}}^n = {\mathcal{J}}^n({\mathcal{J}}^*)^n{\mathcal{J}}^n$ and its adjoint
one can always choose coefficients $a_n$ such that
$
a_n=a_n{\mathcal{J}}^n({\mathcal{J}}^*)^n$, $ n\geq 0 $ and $ a_n={\mathcal{J}}^{-n}({\mathcal{J}}^*)^{-n}a_n$, $ n<0$.

Then, applying the expectation defined in equation \eqref{expect_formula}, we obtain
\begin{equation}\label{ancoeff}
a_n=E(x({\mathcal{J}}^*)^n) ,\quad n\geq 0\qquad \textrm{and}\qquad a_n=E({\mathcal{J}}^{-n}x),\quad n<0,
\end{equation}
and thus such coefficients $a_n$ are unique. More generally, by the above formulas we can associate ``Fourier'' coefficients $a_n\in B_\jcal$ to any $x\in A_\jcal$ and, by the usual Cesaro summation argument~\cite{K}, the coefficients $a_n$ fully determine $x$.

We also have
$\|a_0\|=\|E(x)\|\leq \|x\|$,
which is called the O'Donovan condition.
A key result (see \cite[Theorem~2.7]{LO}) says that since~$A_{\mathcal{J}}$ satisfies the O'Donovan condition then $A_{\mathcal{J}}$ is isomorphic to the crossed product $C^*$-algebra~${B_\jcal\rtimes_{\alpha_\jcal}\N}$
\begin{equation}\label{AJCross}
A_{\mathcal{J}}\cong B_\jcal\rtimes_{\alpha_\jcal}\N,
\end{equation}
see \cite{BKR} and the appendix of~\cite{HKMP2}. For concreteness, we use Stacey's ``multiplicity-one crossed product'', see~\cite{BKR}. Namely,
$B_\fancyJ \rtimes_{\alpha_\fancyJ} \N$ is defined to be the universal $C^*$-algebra generated by a~copy of $B_\fancyJ$ and an isometry $\fancyL$ such that for all $a\in B_\fancyJ$,
$
\fancyL a \fancyL^* = \alpha_\fancyJ(a)$.

The significance of the equation \eqref{AJCross} is that the crossed product on the right-hand side is a~universal object, defined in terms of generators and relations, and thus it gives us a way to construct convenient, concrete representations. A representation of the crossed product ${B_\fancyJ \rtimes_{\alpha_\fancyJ} \N}$ is faithful if it is faithful on $B_\fancyJ$ and satisfies the O'Donovan condition \cite{BKR}. In particular, such a~representation is faithful if it is faithful on $B_\fancyJ$ and admits a gauge action.

\subsection{Invariant ideals}
 In this section, we consider invariant ideals in coefficient algebras and the structure they impose on the crossed products, analogously to the usual crossed products by automorphisms, see \cite[Section~3.4]{Williams}. While this is standard material, for completeness we provide short proofs.

Let $I_\mathcal{J}$ be an $\alpha_\jcal$ and $\beta_\jcal$ invariant ideal in $B_\mathcal{J}$ and let $\mathcal{I}_\jcal$ be the subalgebra in $A_\jcal$ defined as
$
\mathcal{I}_\jcal := C^*(a\jcal^n \mid a\in I_\jcal,\, n=0,1,\dots)$.
Note that $I_\fancyJ \subseteq \fancyI_\fancyJ$. Alternatively, one can define~$\mathcal{I}_\jcal$ by using the Fourier coefficients $a_n$ from equations \eqref{FourierV} and \eqref{ancoeff}
$
x\in \mathcal{I}_\jcal$ if and only if for all~${n\in\Z\colon a_n\in I_\mathcal{J}}$.

We have the following elementary fact.
\begin{Proposition}
The algebra $\mathcal{I}_\jcal$ is an ideal in $A_\jcal$.
\end{Proposition}
\begin{proof}
We write the generators more explicitly $A_\jcal = C^*(b,\jcal, \jcal^*\mid b\in B_\jcal)$ and similarly $\mathcal{I}_\jcal = C^*(a\jcal^n, (\jcal^*)^n a\mid a\in I_\jcal,\, n=0,1,\dots)$. It is sufficient to verify that the products of generators are in $\mathcal{I}_\jcal$. This requires checking multiple cases.

Let $a\in I_\jcal$, $b\in B_\jcal$. It is clear that $ba\jcal^n, a\jcal^n \jcal$, and $\jcal a\jcal^n$ are in $\mathcal{I}_\jcal$.
Next, note that~${a \jcal^n b = a \alpha_\jcal^n(b)\jcal^n\in \mathcal{I}_\jcal}$ since $a\alpha_\jcal^n(b)\in I_\jcal$.

For adjoints, we have
\begin{equation*}
\jcal^* a \jcal^n = \begin{cases}
\jcal^* a \jcal \jcal^{n-1}, & n>0, \\ \jcal^* a, & n=0
\end{cases}
= \begin{cases}
\beta_\jcal(a)\jcal^{n-1}, & n>0, \\ \jcal^*, a & n=0
\end{cases} \in \mathcal{I}_\jcal
\end{equation*}
since $\beta_\fancyJ(a)\in I_\fancyJ$. Further, note that $\jcal \jcal^* = \alpha_\jcal(1) \in B_\jcal$ so
\begin{equation*}
a\jcal^n \jcal^* = \begin{cases}
a \jcal^{n-1}\alpha_\jcal(1), & n>0, \\ a\jcal^*, & n=0
\end{cases} = \begin{cases}
a\alpha_\jcal^n(1)\jcal^{n-1}, & n>0, \\ \jcal^* \alpha_\jcal(a), & n=0
\end{cases} \in \mathcal{I}_\jcal.
\end{equation*}
Finally, $(\jcal^*)^n ab\in \mathcal{I}_\jcal$, $b(\jcal^*)^na = (\jcal^*)^n\alpha_\jcal^n(b)a \in \mathcal{I}_\jcal$,
so $\mathcal{I}_\jcal$ is an ideal completing the proof.\looseness=-1
\end{proof}

Next consider the factor maps $[\alpha_\fancyJ],[\beta_\fancyJ]\colon B_\fancyJ/I_\fancyJ \to B_\fancyJ/I_\fancyJ$ defined by
\begin{equation*}
[\alpha_\fancyJ](a+I_\fancyJ) = \alpha_\fancyJ(a)+I_\fancyJ\qquad\text{and}\qquad [\beta_\fancyJ](a+I_\fancyJ)=\beta_\fancyJ(a)+I_\fancyJ.
\end{equation*}
\begin{Proposition}
$[\alpha_\fancyJ]$ is a monomorphism of $B_\fancyJ/I_\fancyJ$ with hereditary range.
\end{Proposition}
\begin{proof}
Note that $\alpha_\fancyJ$ is an endomorphism of $B_\fancyJ$, thus $[\alpha_\fancyJ]$ is an endomorphism of $B_\fancyJ/I_\fancyJ$. Further, $\beta_\fancyJ \circ \alpha_\fancyJ = I$, so $[\beta_\fancyJ]\circ[\alpha_\fancyJ]=[I]$, thus $[\alpha_\fancyJ]$ is a monomorphism.

For the range of $[\alpha_\fancyJ]$, we know from \cite{HKMP3} that
$
\mathrm{Ran} \alpha_\fancyJ = \alpha_\fancyJ(I)B_\fancyJ \alpha_\fancyJ(I)$.
This implies that
$
 \mathrm{Ran} [\alpha_\fancyJ] = [\alpha_\fancyJ][I]B_\fancyJ/I_\fancyJ [\alpha_\fancyJ][I]
$
and so $[\alpha_\fancyJ]$ has hereditary range.
\end{proof}

We now investigate the structure of the factor algebra $A_\fancyJ/\fancyI_\fancyJ$.
\begin{Proposition}\label{FactorCross}
We have the following isomorphism of algebras:
\begin{equation*}
B_\fancyJ/I_\fancyJ \rtimes_{[\alpha_\fancyJ]} \N \cong A_\fancyJ/\fancyI_\fancyJ.
\end{equation*}
\end{Proposition}
\begin{proof}
$B_\fancyJ \rtimes_{[\alpha_\fancyJ]} \N$ is the universal $C^*$-algebra generated by a copy of $B_\fancyJ/I_\fancyJ$ and an isometry~$\fancyL_\fancyJ$ such that
$
\fancyL_\fancyJ[a]\fancyL_\fancyJ^* = [\alpha_\fancyJ][a]$.

Define the map $\varphi_\fancyJ\colon B_\fancyJ/I_\fancyJ\rtimes_{[\alpha_\fancyJ]}\N \to A_\fancyJ/\fancyI_\fancyJ$ on generators by
\begin{equation*}
\varphi_\fancyJ(a+I_\fancyJ) = a+\fancyI_\fancyJ\qquad\text{and}\qquad \varphi_\fancyJ(\fancyL_\fancyJ) = \fancyJ+\fancyI_\fancyJ.
\end{equation*}
This is well-defined as $I_\fancyJ \subseteq \fancyI_\fancyJ$. Since $\fancyJ$ is an isometry, so is $\varphi_\fancyJ(\fancyL_\fancyJ)$. Furthermore, the map~$\varphi_\fancyJ$ preserves the crossed product relation
\begin{equation*}
\varphi_\fancyJ(\fancyL_\fancyJ)\varphi_\fancyJ(a+I_\fancyJ)\varphi_\fancyJ(\fancyL_\fancyJ)^* = \fancyJ a \fancyJ^* + \fancyI_\fancyJ = \alpha_\fancyJ(a)+\fancyI_\fancyJ = [\alpha_\fancyJ]\varphi_\fancyJ(a+I_\fancyJ).
\end{equation*}
By universality, $\varphi_\fancyJ$ extends to a homomorphism from $B_\fancyJ/I_\fancyJ\rtimes_{[\alpha_\fancyJ]}\N$ to $A_\fancyJ/\fancyI_\fancyJ$.

By uniqueness of the Fourier coefficients of $a$, see equation \eqref{ancoeff}, we see that if $a\in \fancyI_\fancyJ$ with~${a\in B_\fancyJ}$, then $a\in I_\fancyJ$. This means that the map $\varphi_\fancyJ$ is faithful on $B_\fancyJ/I_\fancyJ$. To verify the map is faithful on the whole crossed product, we observe that the algebra $A_\fancyJ/\fancyI_\fancyJ$ admits a gauge action.
Indeed, since $\rho_\theta$ preserves $I_\fancyJ \subseteq B_\fancyJ$ it factors to the quotient $[\rho]_\theta \colon A_\fancyJ/\fancyI_\fancyJ \to A_\fancyJ/\fancyI_\fancyJ$.
We have that $[\rho]_\theta $ is a 1-parameter family of automorphisms with
\begin{equation*}
[\rho]_\theta (a+\fancyI_\fancyJ) = a+\fancyI_\fancyJ \quad \text{for } a\in B_\fancyJ ,\qquad
[\rho]_\theta (\fancyJ+\fancyI_\fancyJ) = {\rm e}^{2\pi {\rm i} \theta}\fancyJ + \fancyI_\fancyJ
\end{equation*}
and therefore $\varphi_\fancyJ$ is faithful by~\cite{BKR}.

Finally, notice that $\varphi_\fancyJ$ is onto because its range is closed and it contains the generators of~$A_\fancyJ/\fancyI_\fancyJ$.
\end{proof}

As a consequence, we get the following short exact sequence of $C^*$-algebras, which is the basis for all structural results in the following sections:
\begin{equation}\label{SESequence}
0\to \mathcal{I}_\mathcal{J}\to A_\mathcal{J}\to B_\fancyJ/I_\fancyJ \rtimes_{[\alpha_\fancyJ]} \N\to 0.
\end{equation}

\section{The Bunce--Deddens shift}
\subsection[The algebra A\_U]{The algebra $\boldsymbol{A_U}$}
In this section, we consider the first shift algebra
$A_U=C^*(U,M_f\mid f\in C(\Z_s))$,
where the shift~$U$ is given on the basis elements of $H=\ell^2(\mathcal{V})$ by $UE_{(n,x)}= E_{(n+1, x+1)}$. The corresponding maps on $A_U$ are
$\alpha_U(a):=U aU^* \in A_U $ and $\beta_U(a):=U^* a U \in A_U$, for $a\in A_U$.

We define the maps $\mathfrak{a}_U$ and $\mathfrak{b}_U$ on $C(\Z_s)$ as follows:
$\mathfrak{a}_Uf(x) = f(x-1)$ and $\mathfrak{b}_Uf(x) = f(x+1) $.
Both are automorphisms on $C(\Z_s)$ and we have
$
\mathfrak{a}_U(1)=\mathfrak{b}_U(1)=1$ and ${\mathfrak{a}_U\circ\mathfrak{b}_U=\mathfrak{b}_U\circ\mathfrak{a}_U=1}$.
The operators $U$ and $M_f$ satisfy the following relations:
\[
M_fU=UM_{\mathfrak{b}_Uf}\qquad \text{and} \qquad \alpha_U(M_f)=\allowbreak M_{\mathfrak{a}_Uf}UU^* .
\]
It follows additionally that
$
UM_f=M_{\mathfrak{a}_Uf}U$ and $ \beta_U(M_f) = M_{\mathfrak{b}_U f}$.

In parallel, we will also describe and analyze the subalgebras
\[
A_{U,n}:=C^*(U,M_{(n,x)}\mid 0\le x < s^n),
\]
where $M_{(n,x)}$ are the multiplication operators by the characteristic functions of balls
\begin{equation*}
M_{(n,x)}E_{(m,y)} =
\begin{cases}
E_{(m,y)} &\textrm{if}\ |y-x|_s \le s^{-n},\\
0 &\textrm{otherwise}.
\end{cases}
\end{equation*}
We have the following formulas:
\begin{equation*}
\mathfrak{a}_UM_{(n,x)} = M_{(n,x+1)}\qquad\text{and}\qquad \mathfrak{b}_UM_{(n,x)} = M_{(n,x-1)} .
\end{equation*}
Notice that $A_{U,n}$ is a subalgebra of $A_{U,n+1}$ and $A_U$ is the $C^*$-algebras direct limit of $A_{U,n}$ via the inclusion maps
$A_U = \smash{\lim\limits_{\longrightarrow} A_{U,n}}$.

\subsection[The coefficient algebra of A\_U]{The coefficient algebra of $\boldsymbol{A_U}$}

Define the mutually orthogonal projections $\{P_n\}_{n\ge 0}$ by
$P_0=I-UU^*$ and $ P_n = U^nP_0(U^*)^{n}$.
By \cite[Proposition 4.2]{HKMP3}, the $C^*$-algebra
\begin{equation*}
B_U:=C^*(M_f,P_n\mid f\in C(\Z_s),\,n\in\Z_{\geq 0})
\end{equation*}
is the smallest coefficient algebra of $A_U$ containing $C(\Z_s)$. The algebra $B_U$ is commutative and by \cite[Proposition 4.3]{HKMP3} every element of $B_U$ can be uniquely written as
\begin{equation*}
T_U(F):= \sum_{n=0}^{\infty}(M_{f_n}-M_{f_\infty})P_n+M_{f_\infty},
\end{equation*}
where $F=(f_n)_{n\in \Z_{\ge0}}\in c(\Z_{\ge0}, C(\Z_s))$, the space of uniformly convergent sequences of continuous functions on $\Z_s$, and
\begin{equation}\label{finf}
C(\Z_s)\ni f_\infty:=\lim_{n\to\infty}f_n.
\end{equation}
We will refer to such a $T_U$ as a Toeplitz map.

Consequently,
$B_U$ is isomorphic to $c(\Z_{\ge0}, C(\Z_s))$.

Put $B_{U,n}:=C^*(P_k,M_{(n,x)}\mid k=0,1,\dots, \, 0\le x < s^n)$. This is an Abelian coefficient algebra for $A_{U,n}$. We can identify $B_{U,n}\cong c(\Z_{\ge 0},\C^{s^n})$, since $C^*(M_{(n,x)}\mid 0\le x < s^n)\cong \C^{s^n}$.

Consider the algebra $I_U$ generated by the products
\begin{equation*}
I_U:=C^*(M_fP_n\mid f\in C(\Z_s),\, n\in\Z_{\geq 0}).
\end{equation*}
Clearly, $I_U$ is an ideal in $B_U$ and, using the above Toeplitz map $T_U$, we see that we have an~identifi\-cation
$I_U\cong c_0(\Z_{\ge0}, C(\Z_s))$,
where $c_0(\Z_{\ge0}, C(\Z_s))$ is consisting of those sequences $F=(f_n)_{n\in \Z_{\ge0}}$ for which $f_\infty=0$.

Similarly, define
\[
I_{U,n}:= C^*(P_kM_{(n,x)} \mid k\in \Z_{\ge 0}, \, 0\le x < s^n) \cong c_0\bigl(\Z_{\ge 0},\C^{s^n}\bigr).
\]
Then $I_{U,n}$ is an ideal in $B_{U,n}$.

Consider the morphisms $\tilde{\alpha}_U$, $\tilde{\beta}_U$ on $c(\Z_{\ge0}, C(\Z_s))$ given by
\begin{equation*}
(\tilde{\alpha}_UF)_n(x) =
\begin{cases}
f_{n-1}(x-1) &\textrm{if }n\ge0, \\
0 &\textrm{if }n=0
\end{cases}\qquad\text{and}\qquad \bigl(\tilde{\beta}_UF\bigr)_n(x) = f_{n+1}(x+1) .
\end{equation*}
A simple calculation shows that we have the following relations:
\begin{equation*}
\alpha_U(T_U(F))=T_U\bigl(\tilde{\beta}_UF\bigr)\qquad\text{and}\qquad \beta_U(T_U(F))=T_U(\tilde{\alpha}_UF) .
\end{equation*}
It follows that the ideals $I_U$ and $I_{U,n}$ are $\alpha_U$ and $\beta_U$ invariant.

We have the following short exact sequence
\begin{equation*}
0\to c_0(\Z_{\ge0}, C(\Z_s))\to c(\Z_{\ge0}, C(\Z_s))\to C(\Z_s)\to 0,
\end{equation*}
where the quotient map is $F\mapsto f_\infty\in C(\Z_s)$.
It follows that the factor algebra $B_U/I_U$ is isomorphic to $C(\Z_s)$
$
B_U/I_U\cong C(\Z_s)$.

The morphisms $\tilde{\alpha}_U$, $\tilde{\beta}_U$ preserve the ideal $c_0(\Z_{\ge0}, C(\Z_s))$ and so they factor to morphisms on~$C(\Z_s)$ which are exactly the automorphisms $\mathfrak{a}_U$ and $\mathfrak{b}_U$. We also get
\[
B_{U,n}/I_{U,n} \cong C^*(M_{(n,x)}) \cong \C^{s^n}.
\]

\subsection[The Structure of A\_U]{The Structure of $\boldsymbol{A_U}$}
The Bunce--Deddens algebra $BD(s^\infty)$ corresponding to the supernatural number $s^\infty$, see \cite{BD1,BD2,KMP2}, is the following crossed product by an automorphism
$BD(s^\infty) = C(\Z_s)\rtimes_{\mathfrak{a}_U}\Z $.
The goal is to relate $A_U$ with the above Bunce--Deddens algebra.

From \cite[Theorem 4.4]{HKMP3}, we have an isomorphism of $C^*$-algebras
$A_U \cong c(\Z_{\ge0}, C(\Z_s))\rtimes_{\tilde{\alpha}_U}\N$.
We also have
$A_{U,n} \cong B_{U,n}\rtimes_{\alpha_U} \N$.

Consider the algebra
\[\mathcal{I}_U := C^*(aU^m \mid a\in I_U,\, m=0,1,\dots).\]
Similarly, we define
\[\mathcal{I}_{U,n} := C^*(aU^m \mid a\in I_{U,n},\, m=0,1,\dots).\]
Then from Section~\ref{sec3}, we know that $\mathcal{I}_U$ is an $\alpha_U$ and $\beta_U$ invariant ideal in $A_U$, while $\mathcal{I}_{U,n}$ is an invariant ideal in~$A_{U,n}$.

There is an alternative, equivalent definition of $\mathcal{I}_U$ which we briefly describe below even though it is not used in subsequent discussion. This alternative description does, however, illustrate some advantages when dealing with universal objects.

Let $\{E_l\}$ be the canonical basis in $\ell^2(\Z)$. Consider the map $c(\Z_{\ge0}, C(\Z_s)) \to B\bigl(\ell^2(\Z)\bigr)$ defined by $F \mapsto m_F$, where $m_F$ is a bounded operator on $\ell^2(\Z)$ given by
$
m_F E_l = f_\infty(l) E_l$.
Additionally, let $u \in B\bigl(\ell^2(\Z)\bigr)$ denote the standard bilateral shift
$
uE_l = E_{l+1}$.
It is easy to check that sending~${U \mapsto u}$ and $M_F \mapsto m_F$ gives rise to a representation of $A_U \cong c(\Z_{\ge0}, C(\Z_s))\rtimes_{\tilde{\alpha}_U}\N$ in $B\bigl(\ell^2(\Z)\bigr)$ since $u$ and $m_F$ satisfy the crossed product relations. Denote this representation by $\pi_\infty$.
It is easily seen that $\mathcal{I}_U$ is the kernel of this representation
$
\mathcal{I}_U = \ker \pi_\infty$.

Next, using a convenient representation, we deduce the structure of the ideal $\mathcal{I}_U$ as described in the following statement.

\begin{Proposition}\label{IUStructure}
Let $\mathcal{K}$ be the $C^*$-algebra of compact operators on a separable Hilbert space.
We get the following isomorphisms of $C^*$-algebras
$
\mathcal{I}_U\cong C(\Z_s)\otimes\mathcal{K}
$
and
$
\fancyI_{U,n}\cong \C^{s^n}\otimes \kcal$.
\end{Proposition}
\begin{proof}
We faithfully represent $A_U$ in $\ell^2(\Z\times\Z_{\ge0})$, which is a more useful representation than the defining representation in $\ell^2(\mathcal{V})$.
If $\{E_{k,l}\}$ is the canonical basis in $\ell^2(\Z\times\Z_{\ge0})$, then the new representation of $A_U$ is given by operators $\mathcal{U}$ and $\mathcal{M}_F$ for functions $F=(f_n)_{n\in \Z_{\ge0}}\in c(\Z_{\ge0}, C(\Z_s))$ given by the following formulas:
\begin{equation*}
\mathcal{U}E_{k,l}=E_{k, l+1} ,\qquad \mathcal{U}^*E_{k,l} =
\begin{cases}
E_{k, l-1} &\textrm{if}\ l\ge1,\\
0 &\textrm{if}\ l=0,
\end{cases}\qquad\text{and}\qquad \mathcal{M}_FE_{k,l} = f_l(k+l)E_{k,l} .
\end{equation*}
Note that $\mathcal{U}$ is the unilateral shift in the second coordinate and these operators satisfy the defining relations for the crossed product $c(\Z_{\ge0}, C(\Z_s))\rtimes_{\tilde{\alpha}_U}\N$. Additionally, the following diagonal unitary operators on $\ell^2(\Z\times\Z_{\ge0})$ given on basis elements by
\begin{equation}\label{repgauge}
\mathbb{U}_\theta E_{k,l}={\rm e}^{2\pi {\rm i}\theta l}E_{k,l}
\end{equation}
with $\theta\in \R/\Z$ satisfy the following relations
$\mathbb{U}_\theta \mathcal{U} \mathbb{U}_\theta^{-1} = {\rm e}^{2\pi {\rm i}\theta}\mathcal{U}$ and $ \mathbb{U}_\theta \mathcal{M}_F \mathbb{U}_\theta^{-1} = \mathcal{M}_F
$
for every~${F\in c(\Z_{\ge0}, C(\Z_s))}$. Thus $\mathbb{U}_\theta$ implements the gauge action in this representation. Since $F\mapsto \mathcal{M}_F$ is faithful, it follows from the previous section that the above representation is a~faithful representation of the crossed product $A_U \cong c(\Z_{\ge0}, C(\Z_s))\rtimes_{\tilde{\alpha}_U}\N$.

For $f\in C(\Z_s)$, let $\tilde f$ be a bounded sequence with values in $C(\Z_s)$ given by the formula $
\tilde f_n(x):=f(x-n)$.
If additionally $c\in c_0(\Z_{\ge0})$, then, by the Stone-Weierstrass theorem, the linear combinations of sequences of the form
$\bigl(\tilde fc\bigr)_n(x)= \tilde f_n(x)c(n)$
are dense in $c_0(\Z_{\ge0}, C(\Z_s))$. Computing on basis elements we get
$\mathcal{M}_{\tilde fc} E_{k,l} = f(k)c(l) E_{k,l}$.
Denoting $\mu_{f} E_{k,l} = f(k) E_{k,l}$ and~${ \nu_{c} E_{k,l} = c(l) E_{k,l}}$,
we see that $\mathcal{M}_{\tilde fc}=\mu_{f}\nu_{c}$. Moreover, for all $f\in C(\Z_s)$ and all $c\in c_0(\Z_{\ge0})$ the operators $\mu_f$ commute with $\nu_{c}$ and with $\mathcal{U}$. Clearly,
$C^*(\mu_f\mid f\in C(\Z_s))\cong C(\Z_s)
$
and, since both $\mathcal{U}$ and $\nu_{c}$ depend only on the second coordinate $l\in\Z_{\ge0}$, we have
\begin{equation*}
C^*(\mathcal{U}^m\nu_{c}\mid m\in \Z_{\ge0},\, c\in c_0(\Z_{\ge0}))\cong \mathcal{K}\bigl(\ell^2(\Z_{\ge0})\bigr),
\end{equation*}
where $\mathcal{K}\bigl(\ell^2(\Z_{\ge0})\bigr)$ is the algebra of compact operators on $\ell^2(\Z_{\ge0})$. Finally, this implies that
\begin{align*}
\mathcal{I}_U&\cong
C^*(\mathcal{U}^m \mathcal{M}_F\mid m\in \Z_{\ge0},\, F\in c_0(\Z_{\ge0}, C(\Z_s))\\
&=C^*(\mathcal{U}^m\nu_{c}\mu_f\mid m\in \Z_{\ge0},\, c\in c_0(\Z_{\ge0}),\, f\in C(\Z_s))\cong C(\Z_s)\otimes \mathcal{K}\bigl(\ell^2(\Z_{\ge0})\bigr).
\end{align*}

Restricting $f$'s in the above argument from $C(\Z_s)$ to the subalgebra $C^*(M_{(n,x)})$, we get the identification $\fancyI_{U,n}\cong \C^{s^n}\otimes \mathcal{K}\bigl(\ell^2(\Z_{\ge0})\bigr)$.
\end{proof}

Consider the quotient $A_U/\mathcal{I}_U$ and the quotient $A_{U,n}/\fancyI_{U,n}$ which, by Proposition \ref{FactorCross}, are isomorphic respectively to the crossed product of $B_U/I_U$ and the crossed product of $B_{U,n}/I_{U,n}$ by the factor monomorphism $[\alpha_U]$. Since we identified $B_U/I_U$ with $C(\Z_s)$ and the endomorphism~$[\alpha_U]$ with $\mathfrak{a}_U$, we get the following immediate consequence of the definition of the Bunce--Deddens algebra:
\begin{Proposition}\label{UAlgebraQuotient}
The Bunce--Deddens algebra $BD(s^\infty)$ is isomorphic to $A_U/\mathcal{I}_U$.
\end{Proposition}

Additionally, using \cite[Proposition 2.9]{KMRSW2}, we get
\begin{equation*}
A_{U,n}/\fancyI_{U,n}\cong C^*(M_{(n,x)})\rtimes_{\mathfrak{a}_u} \Z = BD(s^n)\cong M_{s^n}(\C)\otimes C\bigl(S^1\bigr).
\end{equation*}

The main result of this section in the theorem below is a direct consequence of equation \eqref{SESequence}, Propositions \ref{IUStructure} and \ref{UAlgebraQuotient}.
\begin{Theorem}
We have the short exact sequences of $C^*$-algebras
\begin{gather}
0\rightarrow C(\Z_s)\otimes\mathcal{K}\rightarrow A_U\rightarrow BD(s^\infty)\rightarrow 0,\nonumber
\\
0 \to \C^{s^n}\otimes \mathcal{K}\to A_{U,n} \to M_{s^n}(\C)\otimes C\bigl(S^1\bigr) \to 0.\label{AUExact}
\end{gather}
\end{Theorem}

\subsection[K-theory of A\_U]{$\boldsymbol{K}$-theory of $\boldsymbol{A_U}$}
The results of the previous subsection can be used to describe the $K$-theory of $A_U$, which we do in the following proposition.
\begin{Proposition}
The $K$-theory of $A_U$ satisfies the following: $K_1(A_U) \cong 0$ and
\begin{equation*}
0 \to C(\Z_s, \Z)/\Z \to K_0(A_U) \to G_{s^\infty} \to 0,
\end{equation*}
where $G_{s^\infty} = \bigl\{ \frac{k}{s^n} \in \Q \mid k \in \Z,\, n \in \N \bigr\}$ and the subgroup of $\Z$ in $C(\Z_s,\Z)$ is the subgroup of constant functions.
\end{Proposition}
\begin{proof}
The short exact sequence \eqref{AUExact} above induces the following $6$-term exact sequence in $K$-theory:
\begin{equation*}
\begin{tikzcd}
 K_0(\mathcal{I}_U) \arrow{r} & K_0(A_U) \arrow{r} & K_0(BD(s^\infty)) \arrow{d}{\textup{exp}} \\
 K_1(BD(s^\infty)) \arrow{u}{\textup{ind}} & K_1(A_U) \arrow{l} & K_1(\mathcal{I}_U) \arrow{l}.
\end{tikzcd}
\end{equation*}
Recall that $\mathcal{I}_U \cong C(\Z_s) \otimes \mathcal{K}$. By stability of $K$-theory, along with \cite[Exercise 3.4]{RLL}, and the fact that $\Z_s$ is totally disconnected, we have that
$
K_0(\mathcal{I}_U) \cong C(\Z_s,\Z)$.
Since $C(\Z_s)$ is AF (see \cite[Example III.2.5]{KD}),
$
K_1(\mathcal{I}_U) \cong 0$.
It follows that the exponential map is zero.

By \cite[Exercise 10.11.4\,(b)]{Blackadar}, also see \cite{KMP2},
$K_0(BD(s^\infty)) \cong G_{s^\infty}$ and $ K_1(BD(s^\infty)) \cong \Z$.
The generator of $K_1(BD(s^\infty))$ is a unitary and it lifts to the non-unitary isometry $U$ in $A_U$. By~\cite[Proposition~9.2.4]{RLL}, the index map sends the class of the generator of $K_1(BD(s^\infty))$ to the class of $UU^*-U^*U=-P_0$. The range of the map $\Z \to C(\Z_s , \Z)$ is generated by $-[I - UU^*]_0$, which is seen, by direct calculation, to correspond to the constant function $-1$ in $C(\Z_s, \Z)$. In particular, the kernel of the index map is zero. By exactness on the bottom row of the $6$-term exact sequence, we see that
$
K_1(A_U) \cong 0$.

Since the range of the index map is the subgroup of constant functions in $C(\Z_s, \Z)$, taking the quotient we can extract a short exact sequence of Abelian groups from the top row
\begin{equation*}
0 \to C(\Z_s, \Z)/\Z \to K_0(A_U) \to G_{s^\infty} \to 0.\tag*{\qed}
\end{equation*} \renewcommand{\qed}{}
\end{proof}

The exact sequence above does not split and so does not fully determine $K_0(A_U)$. In fact, as shown below in equation \eqref{star}, we can describe explicitly all the $K_0(A_U)$ classes but the group operation is non-obvious. To overcome this issue, we notice that $A_U$ is the direct limit of~$A_{U,n}$ and so we can, using the continuity of $K$-theory, write $K_0(A_U)$ as a direct limit of easier groups $K_0(A_{U,n})$, which turn out to be powers of $\Z$. We also calculate below the corresponding homomorphisms from $K_0(A_{U,n})$ to $K_0(A_{U,n+1})$.

First we compute the $K$-theory of $A_{U,n}$. This is done by the 6-term exact sequence in $K$-theory as above
\begin{equation*}
\begin{tikzcd}
\Z^{s^n} \cong K_0\bigl(\C^{s^n} \otimes \kcal\bigr) \arrow[r] & K_0(A_{U,n}) \arrow[r] & K_0\bigl(M_{s^n}(\C)\otimes C\bigl(S^1\bigr)\bigr) \cong \Z \arrow{d}{\textup{exp}=0} \\
 \Z \cong K_1\bigl(M_{s^n}(\C)\otimes C\bigl(S^1\bigr)\bigr) \arrow{u}{\textup{ind}} & K_1(A_{U,n})\cong0 \arrow[l] & K_1(\C^{s^n} \otimes \kcal)\cong0 \arrow[l,"0"].
\end{tikzcd}
\end{equation*}
We know that $K_1\bigl(\C^{s^n}\otimes \kcal\bigr)$ is isomorphic to 0, thus the exponential map is 0. We know that the unitary generator of $K_1\bigl(M_{s^n}(\C) \otimes C\bigl(S^1\bigr)\bigr)$ lifts to a~non-unitary in $A_{U,n}$, thus the index map is injective, and exactness of the sequence forces $K_1(A_{U,n})$ to be 0.
Consequently, we get a~simplified short exact sequence
\begin{equation*}
0\to \Z^{s^n}/\Z \to K_0(A_{U,n}) \to \Z \to 0,
\end{equation*}
where the subgroup $\Z$ in the quotient $\Z^{s^n}/\Z$ is the diagonal subgroup. Because the right-most nontrivial group is free, the sequence necessarily splits and thus
\begin{equation*}
K_0(A_{U,n})\cong \bigl(\Z^{s^n}/\Z\bigr) \oplus \Z \cong \Z^{s^n}.
\end{equation*}

The goal is to compute the homomorphisms $\phi_n \colon K_0(A_{U,n}) \to K_0(A_{U,n+1})$ induced by inclusions $A_{U,n}\subseteq A_{U,n+1}$. First notice that $[M_{(n,x)}]$ are minimal projections in
\begin{equation*}
A_{U,n}/\fancyI_{U,n}\cong C^*(M_{(n,x)})\rtimes_{\mathfrak{a}_u} \Z\cong M_{s^n}(\C)\otimes C\bigl(S^1\bigr),
\end{equation*}
and hence their classes give the generator of $K_0\bigl(M_{s^n}(\C)\otimes C\bigl(S^1\bigr)\bigr) \cong \Z$. Specifically, in the formulas below, we choose $[M_{(n,0)}]$ for the generator in $K_0(A_{U,n})$ corresponding to that generator in $K_0(A_{U,n}/\fancyI_{U,n})$.

Also, the projections $M_{(n,x)}$ generate $K_0$ classes of $C^*(M_{(n,x)}) \cong \C^{s^n}$ and, since $P_0$ is a~minimal projection in $\kcal$, the classes $[M_{(n,x)}P_0]$ generate $K_0(\fancyI_{U,n})=K_0(C^*(M_{(n,x)})\otimes \kcal)$. Since~${UU^*=I-P_0}$ and $U^*U=I$, the class of $P_0$ in $K_0(A_{U,n})$ is zero. In the above short exact sequence, this corresponds to the fact that the diagonal subgroup of $K_0(\fancyI_{U,n})$ gives a zero class in $K_0(A_{U,n})$. Thus, we choose $[M_{(n,x)}P_0]$ with $x\ne 0$ for generators of the range of the map~${ \Z^{s^n}/\Z \to K_0(A_{U,n})}$. To summarize, we get the following explicit description of $K_0(A_{U,n})$
\begin{equation}\label{star}
K_0(A_{U,n}) \cong \Biggl\{\sum_{j=1}^{s^n-1} c_j[M_{(n,j)}P_0] + k[M_{(n,0)}] \mid c_j,k\in \Z \Biggr\}.
\end{equation}

With those preliminaries we can now compute the maps $\phi_n \colon K_0(A_{U,n}) \to K_0(A_{U,n+1})$. Notice that the $M_{(n,x)}$'s are indicators for balls in $\Z_s$ so we can decompose these indicators as
\begin{equation*}
M_{(n,x)} = \sum_{l=0}^{s-1} M_{(n+1,x+ls^n)},
\end{equation*}
each of which are mutually orthogonal projections. Thus an element in equation \eqref{star} can be written as
\begin{equation}
\sum_{j=1}^{s^n-1} c_j[M_{(n,j)}P_0] + k[M_{(n,0)}] =
\sum_{j=1}^{s^n-1} \sum_{l=0}^{s-1} c_j [M_{(n+1,j+ls^n)}P_0] + k \sum_{l=0}^{s-1} [M_{(n+1,ls^n)}].\label{starstar}
\end{equation}
To compute $\phi_n$, we need to write the above as combinations of generators of $K_0(A_{U,n+1})$ as in equation~\eqref{star}.

Notice that by Murray--von~Neumann equivalence	
\begin{align*}
[M_{(n+1,0)}] &= \big[U^{ls^n} M_{(n+1,0)} (U^*)^{ls^n}\big] = \big[M_{(n+1,ls^n)} U^{ls^n}(U^*)^{ls^n}\big] \\
&= [M_{(n+1,ls^n)} ] - \big[M_{(n+1,ls^n)}(I-U^{ls^n}(U^*)^{ls^n})\big].
\end{align*}
Thus equation \eqref{starstar} may be written as
\begin{equation}\label{buttholesurfers}
\sum_{j=1}^{s^n-1} \sum_{l=0}^{s-1} c_j [M_{(n+1,j+ls^n)} P_0] + ks[M_{(n+1,0)}] + k \sum_{l=1}^{s-1} \big[M_{(n+1,ls^n)}(I-U^{ls^n}(U^*)^{ls^n})\big].
\end{equation}
Further, notice that for $l>0$
\begin{equation*}
I-U^{ls^n}(U^*)^{ls^n} = P_0 + P_1 + \cdots + P_{ls^n-1}= \sum_{i=0}^{ls^n-1} U^i P_0 (U^*)^i.
\end{equation*}
By using Murray--von~Neumann equivalence, the last sum in equation \eqref{buttholesurfers} is equal to
\begin{equation*}
k\sum_{l=1}^{s-1} \sum_{i=0}^{ls^n-1}\big[M_{(n+1,ls^n)}U^i P_0 (U^*)^i\big]
=k\sum_{l=1}^{s-1} \sum_{i=0}^{ls^n-1}[M_{(n+1,ls^n-i)}P_0]
\end{equation*}
and, reindexing $ls^n-i \to i$, we get equation \eqref{buttholesurfers} as
\begin{equation*}
\sum_{j=1}^{s^n-1} \sum_{l=0}^{s-1} c_j[M_{(n+1,j+ls^n)}P_0] + ks[M_{(n+1,0)}] + k\sum_{l=1}^{s-1} \sum_{i=1}^{ls^n} [M_{(n+1,i)}P_0].
\end{equation*}
Swapping the second double sum in the above equation, we get
\begin{gather*}
\sum_{j=1}^{s^n-1} \sum_{l=0}^{s-1} c_j[M_{(n+1,j+ls^n)}P_0] + ks[M_{(n+1,0)}] + k\sum_{i=1}^{s^n-1} (s-1)[M_{(n+1,i)}P_0]\\
\qquad{} +k\sum_{i=s^n}^{(s-1)s^n} \left(s-\left\lceil\dfrac{i}{s^n}\right\rceil\right) [M_{(n+1,i)}P_0].
\end{gather*}
Reindexing, we write this formula as
\begin{equation*}
\sum_{i=0}^{s^{n+1}-1} c_i'[M_{(n+1,i)}P_0] + k' [M_{(n+1,0)}],
\end{equation*}
where $k'=ks$ and
\begin{equation}\label{coeffs}
c_i' = \begin{cases}
c_{(i\bmod s^n)} + k(s-1), & \text{if } 1\le i < s^n, \\
c_{(i \bmod s^n)} + k\left(s-\left\lceil \dfrac{i}{s^n} \right\rceil\right), & \text{if }s^n \leq i \leq (s-1)s^n ,\\ c_{(i\bmod s^n)}, & \text{if } (s-1)s^n < i < s^{n+1}.
\end{cases}
\end{equation}
Thus the map $\phi_n \colon K_0(A_{U,n}) \to K_0(A_{U,n+1})$, using the identification $K_0(A_{U,n})\cong \Z^{s^n}$ from equation \eqref{star}, is given by
\begin{equation*}
\phi_n((c_1,\dots,c_{s^n-1},k)) = (c_1',\dots,c_{s^{n+1}-1}',k').
\end{equation*}
As a consequence, we get the following proposition.

\begin{Proposition}
The $K$-theory of $A_U$ is given by
$K_0(A_U) = \lim\limits_{\longrightarrow}\Z^{s^n}$
with maps $\phi_n\colon\Z^{s^n} \to \Z^{s^{n+1}}$ given by
\begin{equation*}
\phi_n((c_1,\dots,c_{s^n-1},k)) = (c_1',\dots,c_{s^{n+1}-1}',ks),
\end{equation*}
where $c_i'$ are described in equation~\eqref{coeffs}.
\end{Proposition}

\section{The Hensel shift}
\subsection[The algebra A\_V]{The algebra $\boldsymbol{A_V}$}
In this section, we consider the algebra
$
A_V=C^*(V,M_f\mid f\in C(\Z_s))$,
where the Hensel shift $V$ is given on the basis elements of $H=\ell^2(\mathcal{V})$ by $VE_{(n,x)} = E_{(n+1, sx)}$.
The corresponding maps on $A_V$ are
$
\alpha_V(x):=V x V^* \in A_V$ and $ \beta_V(x):=V^* x V \in A_V$,
for $x\in A_V$.

Consider also the maps: $\mathfrak{a}_V,\mathfrak{b}_V\colon C(\Z_s)\to C(\Z_s)$ given by
\begin{equation*}
\mathfrak{a}_Vf(x) =
\begin{cases}
f\left(\dfrac{x}{s}\right) &\textrm{if }s\mid x,\\
0 &\textrm{else}
\end{cases}\qquad\text{and}\qquad \mathfrak{b}_Vf(x)=f(sx) .
\end{equation*}
Both $\mathfrak{a}_V$ and $\mathfrak{b}_V$ are endomorphisms of $C(\Z_s)$. We have that
$\mathfrak{b}_V\circ\mathfrak{a}_V=1$, $ \mathfrak{b}_V(1)=1$ and $(\mathfrak{a}_V\circ \mathfrak{b}_V)f=\mathfrak{a}_V(1)f$.
Here
\begin{equation*}
\mathfrak{a}_V (1)(x) =
\begin{cases}
1 &\textrm{if }s\mid x, \\
0 &\textrm{else}
\end{cases}
\end{equation*}
is the characteristic function of the subset of $\Z_s$ of elements divisible by $s$.

Given $f\in C(\Z_s)$, $V$ and $M_f$ satisfy the following relations
$
M_fV=VM_{\mathfrak{b}_Vf}$ and $ \alpha_V(M_f)=M_{\mathfrak{a}_Vf}(I-P_{(0,0)}) $,
where $P_{(0,0)}$ is the projection onto (the subspace generated by) the basis element~$E_{(0,0)}$ in $H$.
We also have
$
VM_f=M_{\mathfrak{a}_Vf}V $ and $ \beta_V(M_f) = M_{\mathfrak{b}_V f}$.

\subsection[The coefficient algebra of A\_V]{The coefficient algebra of $\boldsymbol{A_V}$}
If we define
$
P_{(n,0)} = V^nP_{(0,0)}(V^*)^{n}$ for $n\in\Z_{\ge0}$,
we get $P_{(n,0)}\in A_V$ for all $n\in\Z_{\ge0}$. Additionally, $P_{(n,0)}$ is the one-dimensional projection onto (the subspace generated by) the basis element $E_{(n,0)}$.

By \cite[Proposition 5.2]{HKMP3}, the $C^*$-algebra $B_V$ generated by $M_f$'s and $P_{(n,0)}$
\begin{equation*}
B_V:=C^*(M_f,P_{(n,0)}\mid f\in C(\Z_s),\,n\in\Z_{\geq 0})
\end{equation*}
is the smallest coefficient algebra of $A_V$ containing $C(\Z_s)$.
 It easily follows from the above definitions that we have the following relation:
 \begin{equation}\label{Vproductrelation}
M_fP_{(n,0)}=f(0)P_{(n,0)} .
\end{equation}
Consequently, $B_V$ is commutative.

Consider the algebra $I_V$ generated by the projections $P_{(n,0)}$,
$I_V:= C^*(P_{(n,0)}, n\in\Z_{\geq 0})$.
By equation \eqref{Vproductrelation}, it is an ideal in $B_V$ and we have
$I_V\cong c_0(\Z_{\ge 0})$.
Since $\alpha_V(P_{(n,0)})=P_{(n+1,0)}$ and~${\beta_V(P_{(n,0)})=P_{(n-1,0)}}$, with $P_{(-1,0)}=0$, we see that the ideal $I_V$ is $\alpha_V$ and $\beta_V$ invariant.

 By \cite[Proposition 5.3]{HKMP3}, every element of $B_V$ can be uniquely written as
\begin{equation*}
 T_V(F):= \sum_{n=0}^{\infty}(x_n-f(0))P_{(n,0)}+M_{f} \in B_V,
\end{equation*}
where $F=(f, (x_n)_{n=0}^\infty)$ with $f\in C(\Z_s)$, and $(x_n)_{n=0}^\infty$ is a sequence such that $\lim_{n\to\infty}x_n=f(0)$. It follows that the factor algebra $B_V/I_V$ is isomorphic to $C(\Z_s)$
$
B_V/I_V\cong C(\Z_s)$.

By \cite[equation (5.2)]{HKMP3}, we have the following relations:
\begin{equation*}
V^*T_V(F)V=T_V\bigl(\tilde{\beta}_VF\bigr)\qquad\text{and}\qquad VT_V(F)V^*=T_V(\tilde{\alpha}_VF) ,
\end{equation*}
where the morphisms $\tilde{\alpha}_V$, $\tilde{\beta}_V$ are given by
\begin{equation*}
\tilde{\alpha}_V(f,(x_n)_{n=0}^\infty)= (\mathfrak{a}_V f,(x_{n-1})_{n=0}^\infty)\qquad\text{and}\qquad \tilde{\beta}_V(f,(x_n)_{n=0}^\infty) = (\mathfrak{b}_V f,(x_{n+1})_{n=0}^\infty),
\end{equation*}
taking $x_{-1}=0$.
It follows that the factor endomorphism $[\alpha_V]$ on $B_V/I_V\cong C(\Z_s)$ is identified with the endomorphism $\mathfrak{a}_V$ on $C(\Z_s)$.

\subsection[The structure of A\_V]{The structure of $\boldsymbol{A_V}$}

In \cite{HKMP2}, the authors defined the Hensel--Steinitz algebra ${\rm HS}(s)$ as the following crossed product by an endomorphism
${\rm HS}(s) = C(\Z_s)\rtimes_{\mathfrak{a}_V}\N$.
Here, since $\mathfrak{a}_V$ is a monomorphism with hereditary range, we can again use Stacey's definition of the crossed product in \cite{St} as the universal $C^*$-algebra generated by a copy of $C(\Z_s)$ and an isometry $v$ satisfying
$vfv^*=\mathfrak{a}_V(f)$. The goal is to relate $A_V$ with the above Hensel--Steinitz algebra.

Let $\mathcal{I}_V := C^*(aV^n \mid a\in I_V,\, n=0,1,\dots)$. Then from Section~\ref{sec3}, we know that $\mathcal{I}_V$ is an $\alpha_V$ and $\beta_V$ invariant ideal in $A_V$.
We have the following observation.

\begin{Proposition}\label{IVStructure}
The $C^*$-algebra $\mathcal{I}_V$ is isomorphic to the $C^*$-algebra of compact operators on a separable Hilbert space
$
\mathcal{I}_V\cong\mathcal{K}$.
\end{Proposition}
\begin{proof}
If we define
$\tilde P_{(k,l)} = V^kP_{(0,0)}(V^*)^{l}$ for $k,l\in\Z_{\ge0}$,
we get \smash{$\tilde P_{(k,l)}\in A_V$} for all ${k,l\in\Z_{\ge0}}$. Clearly, we have
\smash{$\mathcal{I}_V=C^*\bigl(\tilde P_{(k,l)}\mid k,l\in\Z_{\ge0}\bigr)$}.
A straightforward calculation yields that
$
\smash{\tilde P_{(k,l)}^*}=\smash{\tilde P_{(l,k)}}$ and \smash{$\tilde P_{(k,l)}\tilde P_{(m,n)}=\delta_{lm}\tilde P_{(k,n)}$}.
Thus \smash{$\{\tilde P_{(k,l)}\}_{k,l\in\Z_{\ge0}}$} form a system of matrix units and therefore $\mathcal{I}_V$ is isomorphic to the $C^*$-algebra of compact operators on a separable Hilbert space.
\end{proof}

Consider the quotient $A_V/\mathcal{I}_V$, which by Proposition \ref{FactorCross} is isomorphic to the crossed product of $B_V/I_V$ by the factor monomorphism $[\alpha_V]$ with hereditary range. Since we identified $B_V/I_V$ with $C(\Z_s)$ and the endomorphism $[\alpha_V]$ with $\mathfrak{a}_V$, we get the following immediate consequence of the definition of the Hensel--Steinitz algebra.

\begin{Proposition}\label{VAlgebraQuotient}
The Hensel--Steinitz algebra ${\rm HS}(s)$ is isomorphic to $A_V/\mathcal{I}_V$.
\end{Proposition}

The main result of this section in the theorem below is a direct consequence of equation \eqref{SESequence}, Propositions~\ref{IVStructure} and~\ref{VAlgebraQuotient}.
\begin{Theorem}
We have the short exact sequence of $C^*$-algebras
\begin{equation}\label{AVExact}
0 \to\mathcal{K} \to A_V \to {\rm HS}(s) \to 0 .
\end{equation}
\end{Theorem}

\subsection[K-theory of A\_V]{$\boldsymbol{K}$-theory of $\boldsymbol{A_V}$}
The $K$-theory of $A_V$ can be computed from the $K$-theory of ${\rm HS}(s)$, along with the short exact sequence involving $A_V$ and ${\rm HS}(s)$. The $K$-theory of ${\rm HS}(s)$ was computed in \cite{HKMP2} and is given by~${
K_0({\rm HS}(s)) \cong C(\Z^\times_s, \Z)}$ and $ K_1({\rm HS}(s)) \cong 0$,
where $C(\Z_s^\times, \Z)$ denotes the continuous functions on $\Z_s^\times$ with values in $\Z$ where $\Z_s^\times$ is the unit $s$-adic sphere, the set of $x\in \Z_s$ so that~${|x|_s = 1}$. This leads to the following result.
\begin{Proposition}
The $K$-theory of $A_V$ is given by
$
K_0(A_V) \cong C(\Z^\times_s, \Z) \oplus \Z $ and $ K_1(A_V) \cong 0$.
\end{Proposition}
\begin{proof}
The short exact sequence from equation \eqref{AVExact}
induces the following $6$-term sequence in $K$-theory
\begin{equation*}
\begin{tikzcd}
 K_0(\mathcal{I}_V) \arrow{r} & K_0(A_V) \arrow{r} & K_0({\rm HS}(s)) \arrow{d}{\textup{exp}} \\
 K_1({\rm HS}(s)) \arrow{u}{\textup{ind}} & K_1(A_V) \arrow{l} & K_1(\mathcal{I}_V) \arrow{l}.
\end{tikzcd}
\end{equation*}
Since $\mathcal{I}_V$ is isomorphic to the $C^*$-algebra of compact operators, $K_0(\mathcal{I}_V) \cong \Z$ and $K_1(\mathcal{I}_V) \cong 0$. Since $K_1({\rm HS}(s)) \cong 0$, it follows immediately from exactness at $K_1(A_V)$ that
$K_1(A_V) \cong 0$.
We can then extract a short exact sequence of Abelian groups from the top row
\begin{equation*}
0 \to \Z \to K_0(A_V) \to C(\Z_s^\times, \Z) \to 0.
\end{equation*}
\cite[Exercise 7.7.5]{HR} shows that $C(\Z_s^\times, \Z)$ is free, thus this sequence necessarily splits. Hence,
$
K_0(A_V) \cong C(\Z_s^\times , \Z) \oplus \Z$.
This completes the proof.
\end{proof}

\section{The Bernoulli shift}
\subsection[The algebra A\_S]{The algebra $\boldsymbol{A_S}$}
Recall that the Bernoulli shift $S\colon H\to H$ is defined as
\begin{equation*}
SE_{(n,x)}=\frac{1}{s}\sum_{j=0}^{s-1}E_{(n+1, sx+j)} .
\end{equation*}
In this section, we study the algebra
$
A_S=C^*(S,M_f\mid f\in C(\Z_s))$.
The corresponding maps on~$A_S$ are
$
\alpha_S(a):=S aS^* \in A_S
$ and $ \beta_S(a):=S^* a S \in A_S$,
for $a\in A_S$.

Like in the other cases we have maps $\mathfrak{a}_S,\mathfrak{b}_S\colon C(\Z_s)\to C(\Z_s)$ defined by
\begin{equation*}
\mathfrak{a}_Sf(x) = f\left(\frac{x-x\textrm{ mod }s}{s}\right)\qquad\text{and}\qquad \mathfrak{b}_Sf(x) = \frac{1}{s}\sum_{j=0}^{s-1}f(sx+j) .
\end{equation*}
It is clear that $\mathfrak{a}_S$ is an endomorphism but $\mathfrak{b}_S$ is not a homomorphism as it was before, however~$\mathfrak{b}_S$ is positive and linear. We also have that $\mathfrak{b}_S(I)=I=\mathfrak{a}_S(I)$. It is also straightforward to see that $\mathfrak{a}_S$ has trivial kernel.
We also have that $\mathfrak{b}_S\circ\mathfrak{a}_S=I$ and that $\Ran\mathfrak{b}_S=C(\Z_s)$.

Given $f\in C(\Z_s)$, $S$ and $M_f$ satisfy the following relations
$\beta_S(M_f)=M_{\mathfrak{b}_Sf} $ and $ SM_f=M_{\mathfrak{a}_Sf}S$
but, unlike for the previous to shifts, $M_fS$ is not equal to $SM_{\mathfrak{b}_Sf}$.

For $j=0,1,\dots, s-1$ let $\chi_j$ be the characteristic function for the ball in $\Z_s$ with center $j$ and radius $s^{-1}$
\begin{equation*}
\chi_j(x) =
\begin{cases}
1 &\textrm{if }x \textrm{ mod }s=j, \\
0 &\textrm{else.}
\end{cases}
\end{equation*}
Notice that for each $j$, $\chi_j$ is locally constant and hence is a continuous function on $\Z_s$, see \cite{Ro}. We define operators $S_j\colon H\to H$ for each $j=0,\dots, s-1$ by
\begin{equation}\label{S_j defn}
S_j=sM_{\chi_j}S.
\end{equation}
By \cite[Proposition 6.2]{HKMP3}, we have the equality of $C^*$-algebras
$
A_S=C^*(S_j\mid j=0,\dots,s-1)$.
Two key relations between the $S_j$'s, are the following:
\begin{equation}\label{A_S_cuntz_rel}
S_j^*S_k=\delta_{jk}I\quad \text{for all}\ j,k=0,\dots,s-1 \qquad \text{and}\qquad\sum_{j=0}^{s-1}S_jS_j^*=I-P_{(0,0)} ,
\end{equation}
where $P_{(0,0)}$ is the projection onto (the subspace generated by) the basis element $E_{(0,0)}$. This identifies $A_S$ with the Cuntz--Toeplitz algebra, see \cite{HKMP3}.

It is clear that $P_{(0,0)}$ is a diagonal compact operator, and in fact by \cite[Proposition 6.3]{HKMP3}, all compact operators on $H$ are in $A_S$.

\subsection[The coefficient algebra of A\_S]{The coefficient algebra of $\boldsymbol{A_S}$}
Since Cuntz and Cuntz--Toeplitz algebras are textbook material, we skip some of the arguments below. However, for completeness we decided to include all the details organized so that they fit the discussions of other shifts.

For $n\in \Z_{\ge0}$, $0 \le x < s^n$, expand $x$ as
$x = x_0 + x_1 s + x_2 s^2 + \cdots + x_{n-1} s^{n-1}$,
with the coefficients satisfying $0\leq x_k<s$.
Define the operator $S_{(n,x)}$ on $H$ by
\[
S_{(n,x)} := S_{x_0}S_{x_1}\cdots S_{x_{n-1}},
\]
where $S_{x_k}$ is given by equation \eqref{S_j defn} for $k=0,\dots,n-1$.
By \cite[Proposition 6.5]{HKMP3}, the following subalgebra of $A_S$
\begin{equation*}
B_S:= C^*(S_{(n,x)}S_{(n,y)}^*\mid n\in \Z_{\ge0},\, 0\le x,y<s^n).
\end{equation*}
 is the minimal coefficient algebra extension of $C(\Z_s)$.

It was noticed in \cite[Section 6.6]{HKMP3} that the algebra $\mathcal{K}(H)_{\mathrm{inv}}$ of gauge invariant compact operators is a subalgebra, in fact an ideal of $B_S$. Thus, we take
$I_S:=\mathcal{K}(H)_{\mathrm{inv}}$
and we see from the definitions that the ideal $I_S$ is $\alpha_S$ and $\beta_S$ invariant.

 Let $\mathcal{O}_s$ be the Cuntz algebra indexed by $s\ge2$. This algebra is the universal $C^*$-algebra with generators $\{u_j\}_{j=0}^{s-1}$ and relations
$u_j^*u_k=\delta_{jk}$ and $ \sum_{j=0}^{s-1}u_ju_j^*=1$.
As before, for $n\in \Z_{\ge0}$, $0 \le x < s^n$, expand $x$ as
\smash{$
x = x_0 + x_1 s + x_2 s^2 + \cdots + x_{n-1} s^{n-1}$},
and define the following elements of $\mathcal{O}_s$,
$
u_{(n,x)} = u_{x_0}u_{x_1}\cdots u_{x_{n-1}}$.

Consider the following subalgebra of $\mathcal{O}_s$:
\begin{equation*}
\mathcal{O}_{s,\mathrm{inv}}:= C^*(u_{(n,x)}u_{(n,y)}^*\mid n\in \Z_{\ge0},\, 0\le x,y<s^n).
\end{equation*}
It is known, see \cite{Cu3}, that $\mathcal{O}_{s,\mathrm{inv}}$ is the invariant subalgebra of $\mathcal{O}_s$ with respect to the gauge action on $\mathcal{O}_s$ and that
$\mathcal{O}_{s,\mathrm{inv}}\cong UHF(s^\infty)$,
the uniformly hyperfinite $C^*$-algebra corresponding to the supernatural number $s^\infty$.

\cite[Section 6.4]{HKMP3} gives a definition of the Toeplitz map below and describes very useful properties of it
$
T_S\colon \mathcal{O}_s\to A_S$.
Proposition 6.6 of that paper establishes that every $b\in B_S$ can be uniquely written as
$
b=T_S(a)+c$,
where $a\in \mathcal{O}_{s,\mathrm{inv}}$ and $c\in \mathcal{K}(H)_{\mathrm{inv}}$. Consequently,
we have~${
B_S/I_S\cong \mathcal{O}_{s,\mathrm{inv}} \cong UHF(s^\infty)}$.

\subsection[The structure of A\_S]{The structure of $\boldsymbol{A_S}$}
Let $\mathcal{I}_S := C^*(aS^n \mid a\in I_S,\, n=0,1,\dots)$. Then from Section~\ref{sec3} we know that $\mathcal{I}_S$ is an $\alpha_S$ and~$\beta_S$ invariant ideal in $A_S$. Since $\mathcal{K}(H)$ is the only proper nontrivial ideal in the Cuntz--Toeplitz algebra $A_S$, we have
$
\mathcal{I}_S\cong \mathcal{K}(H)$.

To summarize, our key relations in equation \eqref{A_S_cuntz_rel} imply that $A_S$ is the Cuntz--Toeplitz algebra. By the structure of the Cuntz--Toeplitz algebra, see \cite[Lemma V.5.2]{KD},
we get the following result~${
A_S/\mathcal{I}_S\cong A_S/\mathcal{K}(H)\cong\mathcal{O}_s}$.
This implies that we have the following short exact sequence
\begin{equation}\label{ses-As}
0\rightarrow\mathcal{K}(H)\rightarrow A_S\rightarrow\mathcal{O}_s\rightarrow 0.
\end{equation}

\subsection[K-theory of A\_S]{$\boldsymbol{K}$-theory of $\boldsymbol{A_S}$}
The short exact sequence in equation \eqref{ses-As} also induces a $6$-term exact sequence in $K$-theory. Since the $K$-theory and generators of $\mathcal{O}_s$ and $\mathcal{K}(H)$ are known, this can be used to compute the $K$-theory and generators of $A_S$. These results are summarized in the following proposition.
\begin{Proposition}
The $K$-theory of $A_S$ is given by
$
K_0(A_S) \cong \Z$ and $ K_1(A_S) \cong 0$.
\end{Proposition}
\begin{proof}
The short exact sequence
\begin{equation*}
0\rightarrow\mathcal{K}(H)\rightarrow A_S\rightarrow\mathcal{O}_s\rightarrow 0
\end{equation*}
induces the following $6$-term exact sequence in $K$-theory:
\begin{equation*}
\begin{tikzcd}
 K_0(\mathcal{K}(H)) \arrow{r} & K_0(A_S) \arrow{r} & K_0(\mathcal{O}_s) \arrow{d}{\textup{exp}} \\
 K_1(\mathcal{O}_s) \arrow{u}{\textup{ind}} & K_1(A_S) \arrow{l} & K_1(\mathcal{K}(H)) \arrow{l}.
\end{tikzcd}
\end{equation*}
Since $K_1(\mathcal{O}_s) \cong 0$ (see \cite{Cu2}), and $K_1(\mathcal{K}(H)) \cong 0$, it follows immediately that
$K_1(A_S) \cong 0$.
From the $6$-term exact sequence, we can extract a short exact sequence containing $K_0(A_S)$,
\begin{equation*}
0 \to K_0(\mathcal{K}(H)) \to K_0(A_S) \to K_0(\mathcal{O}_s) \to 0.
\end{equation*}
Recall that $K_0(\mathcal{K}(H)) \cong \Z$, and is generated by the class of any rank $1$ projection, such as
\[
[P_{(0,0)}]_0 = \left[I - \sum_{i=0}^{s-1} S_iS_i^*\right]_0.
\]
On the other hand, we have
$K_0(\mathcal{O}_s) \cong \Z / (s-1) \Z$,
and it is generated by the class of the identity~$[I]_0$, see \cite{Cu2}. $K_0(A_S)$ is then necessarily generated by $[I]_0$ and $\big[I - \sum_{i=0}^{s-1} S_iS_i^*\big]_0$ (as classes of projections in $A_S$). However, it is easy to see that the class \smash{$\big[\sum_{i=0}^{s-1} S_iS_i^*\big]_0$} coincides with~$s [I]_0$, via Murray--von~Neumann equivalence using the element $\textup{diag}(S_0, \dots , S_{s-1})$ of $M_s(A_S)$, and the fact that $S_iS_i^*$ and $S_jS_j^*$ are mutually orthogonal for $i \neq j$. From this, we have that
\begin{equation*}
\left[I - \sum_{i=1}^{s-1} S_iS_i^*\right]_0 = -(s-1)[I]_0.
\end{equation*}
Hence $K_0(A_S)$ is generated by $[I]_0$. If $[I]_0$ had finite order, then the map
$K_0(\mathcal{K}(H)) \to K_0(A_S)$
would not be injective, which contradicts exactness. Hence $K_0(A_S) \cong \Z$ and is generated by~$[I]_0$. This completes the proof.
\end{proof}

\section{The Serre shift}
\subsection[The algebra A\_W]{The algebra $\boldsymbol{A_W}$}

In this final section, the goal is to study the algebra
$A_W=C^*(W,M_f\mid f\in C(\Z_s))$.
 Recall that the Serre shift is given by
\begin{equation*}
WE_{(n,x)} = \frac{1}{\sqrt{s}}\sum_{j=0}^{s-1}E_{(n+1, x+js^n)} .
\end{equation*}

Consider the algebra $\mathcal{F}$ of functions on $\mathcal{V}$ defined as
\begin{gather*}
\mathcal{F}:=\{F\colon \mathcal{V}\to\C \mid \textrm{there is}\,f_F\in C(\Z_s), \, \textrm{so that}\,\lim_n\underset{0\le x<s^n}{\textrm{sup }}|F(n,x\textrm{ mod }s^n)-f_F(x)|=0\} .
\end{gather*}
If $F\in\mathcal{F}$, we define
 the operator $M_F$ acting on $\ell^2(\mathcal{V})$ by
$M_FE_{(n,x)}=F(n,x)E_{(n,x)}$.
By \cite[Proposition~7.2]{HKMP3}, for every $F\in\mathcal{F}$ we have $M_F\in A_W$. Additionally, by \cite[Proposition~7.3]{HKMP3}, the $C^*$-algebra $C^*(W,M_F)$ generated by $W$ and $M_F$ with $F\in\mathcal{F}$ is equal to $A_W$.

 If $F\in\mathcal{F}$, we define
\begin{gather*}
\mathfrak{a}_WF(n,x) =
\begin{cases}
F\bigl(n-1, x\textrm{ mod }s^{n-1}\bigr) &\textrm{if }n\ge1,\\
0 &\textrm{else},
\end{cases}\\
\mathfrak{b}_WF(n,x) = \frac{1}{s}\sum_{j=0}^{s-1}F(n+1, x+js^n) .
\end{gather*}
With those definitions, we get the following useful relations
$
\beta_W(M_F)=M_{\mathfrak{b}_WF}$ and $ WM_F=M_{\mathfrak{a}_WF}W$.

\subsection[The coefficient algebra of A\_W]{The coefficient algebra of $\boldsymbol{A_W}$}
Consider the mutually orthogonal projections $\{\mathcal{P}_n\}_{n\in \Z_{\ge 0}}$ given by
$\mathcal{P}_0=I-WW^*$ and $ \mathcal{P}_n = W^n\mathcal{P}_0(W^*)^{n}$.
By \cite[Proposition 7.8]{HKMP3}, we know that the $C^*$-algebra $B_W$ generated by $M_F$'s and $\{\mathcal{P}_n\}$
\begin{equation*}
B_W:=C^*(M_F,\mathcal{P}_n\mid F\in\mathcal{F}, \, n\in\Z_{\geq 0})
\end{equation*}
is a coefficient algebra of $A_W$.

It was noticed in \cite[Propositions 7.4 and 7.8]{HKMP3} that the algebra $\mathcal{K}(H)_{\mathrm{inv}}$ of gauge invariant compact operators is a subalgebra, in fact an ideal of $B_W$. Thus, we take
$
I_W:=\mathcal{K}(H)_{\mathrm{inv}}
$
and we see from the definitions that the ideal $I_W$ is $\alpha_W$ and $\beta_W$ invariant.

The space $c(\Z_{\ge0}, C(\Z_s))$ of uniformly convergent sequences of continuous functions on $\Z_s$ appeared in the discussion of the algebra $A_U$ and it also is needed here. Consider the following Toeplitz-like map:
\begin{equation*}
c(\Z_{\ge0}, C(\Z_s))\ni G\mapsto T_W(G):= \sum_{n=0}^{\infty}\mathcal{P}_n(M_{g_n}-M_{g_\infty})\mathcal{P}_n+M_{g_\infty} \in B_W,
\end{equation*}
where $g_\infty$ is defined in equation \eqref{finf}.
Then, by \cite[Proposition 7.7]{HKMP3} every $a\in B_W$ can be uniquely written as
$a=T_W(G)+c$,
for some $c\in \mathcal{K}(H)_{\mathrm{inv}}$ and $G\in c(\Z_{\ge0}, C(\Z_s))$. It follows, see also \cite[Proposition 7.8]{HKMP3}, that the factor algebra $B_W/I_W$ is isomorphic to $c(\Z_{\ge0}, C(\Z_s))$,
$
B_W/I_W\cong c(\Z_{\ge0}, C(\Z_s))$.

By \cite[Proposition 7.8]{HKMP3}, we have the following relation
$
\alpha_W(T_W(G))=T_W(G')+c$,
where~$c$ is compact and $G'=(g_n')_{n\in \Z_{\ge0}}$ with $g_n'=g_{n-1}$ and
taking $g_{-1}=0$.
It follows that the factor endomorphism $[\alpha_W]$ on $B_W/I_W\cong c(\Z_{\ge0}, C(\Z_s))\cong c(\Z_{\ge0}) \otimes C(\Z_s)$ is identified with endomorphism $\alpha\otimes I$ on $c(\Z_{\ge0}) \otimes C(\Z_s)$ where $\alpha\colon c(\Z_{\ge0})\to c(\Z_{\ge0})$ is a monomorphism given by\looseness=-1
\begin{equation}\label{shiftdef}
\alpha((x_n)_{n=0}^\infty)= (x_{n-1})_{n=0}^\infty
\end{equation}
with $x_{-1}:=0$.

\subsection[The structure of A\_W]{The structure of $\boldsymbol{A_W}$}
The classical and motivating example of a crossed product by an endomorphism is the Toeplitz algebra \cite{Cob}
$
\mathcal{T}=c(\Z_{\ge0})\rtimes_{\alpha}\N
$
with $\alpha$ defined in equation \eqref{shiftdef}. The goal is to relate $A_W$ with the above Toeplitz algebra.

Let $\mathcal{I}_W := C^*(aW^n \mid a\in I_W,\, n=0,1,\dots)$. Then from Section \ref{sec3}, we know that $\mathcal{I}_W$ is an $\alpha_W$ and $\beta_W$ invariant ideal in $A_W$.
We have the following observation.
\begin{Proposition}\label{IWStructure}
The $C^*$-algebra $\mathcal{I}_W$ is equal to the $C^*$-algebra of compact operators on~$H$
\begin{equation*}
\mathcal{I}_W=\mathcal{K}(H).
\end{equation*}
\end{Proposition}
\begin{proof} Clearly, $\mathcal{I}_W\subseteq \mathcal{K}(H)$. Consider the matrix units for the canonical base of $H$
\begin{equation*}
P_{(n,x),(m,y)} E_{(k,z)} = \begin{cases}
E_{(n,x)}, & (m,y)=(k,z) ,\\ 0, & \text{else}.
\end{cases}
\end{equation*}
They generate $\mathcal{K}(H)$ and by \cite[Proposition 7.4]{HKMP3}, $\mathcal{K}(H)\subseteq A_W$, so $P_{(n,x),(m,y)} \in A_W$.
Computing the gauge action on the matrix units gives
\begin{equation*}
\rho_\theta(P_{(n,x),(m,y)} )={\rm e}^{2\pi {\rm i}\theta(n-m)}P_{(n,x),(m,y)} .
\end{equation*}
If $n> m$, we can write
\begin{equation*}
P_{(n,x),(m,y)} =P_{(n,x),(m,y)} (W^*)^{n-m}W^{n-m}
\end{equation*}
and notice that $P_{(n,x),(m,y)} (W^*)^{n-m}$ is gauge invariant and compact. A similar argument works for $n< m$. It follows that $P_{(n,x),(m,y)}\in \mathcal{I}_W$.
\end{proof}

We now consider the factor algebra $A_W/\mathcal{I}_W$, which is isomorphic to the crossed product of~$B_W/I_W\cong c(\Z_{\ge0}) \otimes C(\Z_s)$ by the monomorphism $\alpha\otimes I$.

\begin{Proposition}\label{WAlgebraQuotient}
We have an isomorphism of $C^*$-algebras
$A_W/\mathcal{I}_W\cong \mathcal{T}\otimes C(\Z_s)$,
where $\mathcal{T}$ is the Toeplitz algebra.
\end{Proposition}
\begin{proof} This result can be inferred from general considerations on tensor products of crossed products. However, we give a direct argument specific to our situation slightly modifying the arguments from Proposition~\ref{IUStructure}.

If $\{E_{k,l}\}$ is the canonical basis in $\ell^2(\Z\times\Z_{\ge0})$, then a faithful representation of the crossed product $c(\Z_{\ge0}, C(\Z_s))\rtimes_{\alpha\otimes I}\N$ is given by operators $\mathcal{U}$ and \smash{$\widehat M_F$} for functions $F=(f_n)_{n\in \Z_{\ge0}}\in c(\Z_{\ge0}, C(\Z_s))$ given by the following formulas:
\begin{equation*}
\mathcal{U}E_{k,l}=E_{k, l+1} ,\qquad \mathcal{U}^*E_{k,l} =
\begin{cases}
E_{k, l-1} &\textrm{if }l\ge1,\\
0 &\textrm{if }l=0,
\end{cases}\qquad\text{and}\qquad \widehat M_FE_{k,l} = f_l(k)E_{k,l} .
\end{equation*}
It is easy to verify that these operators satisfy the defining relations for the crossed product. Additionally, as before, the operators $\mathbb{U}_\theta$ from equation \eqref{repgauge} implement the gauge action in this representation. Since \smash{$F\mapsto \widehat M_F$} is faithful, it follows that the above representation is a faithful representation of the crossed product $c(\Z_{\ge0}, C(\Z_s))\rtimes_{\alpha\otimes I}\N$.

Since $c(\Z_{\ge0}, C(\Z_s))\cong c(\Z_{\ge0}) \otimes C(\Z_s)$, finite linear combinations of functions of the form $fc$ with $f\in C(\Z_s)$ and $c\in c(\Z_{\ge0})$ are dense in $c(\Z_{\ge0}) \otimes C(\Z_s)$.
We get a factorization
\begin{equation*}
 \widehat M_{fc} E_{k,l} = f(k)c(l) E_{k,l}.
\end{equation*}
Denoting
$\mu_{f} E_{k,l} = f(k) E_{k,l}$ and $ \nu_{c} E_{k,l} = c(l) E_{k,l}$,
we see that \smash{$\widehat M_{fc}=\mu_{f}\nu_{c}$}. Moreover, for all~${f\in C(\Z_s)}$ and all $c\in c(\Z_{\ge0})$ the operators $\mu_f$ commute with $\nu_{c}$ and with $\mathcal{U}$. Clearly,
\begin{equation*}
C^*(\mu_f\mid f\in C(\Z_s))\cong C(\Z_s)
\end{equation*}
and, since both $\mathcal{U}$ and $\nu_{c}$ depend only on the second coordinate $l\in\Z_{\ge0}$, we have
\begin{equation*}
C^*(\mathcal{U}^m\nu_{c}\mid m\in \Z_{\ge0},\, c\in c(\Z_{\ge0}))\cong \mathcal{T}\bigl(\ell^2(\Z_{\ge0})\bigr),
\end{equation*}
where $\mathcal{T}\bigl(\ell^2(\Z_{\ge0})\bigr)$ is the Toeplitz algebra of operators in $B\bigl(\ell^2(\Z_{\ge0})\bigr)$. Finally, this implies that
\begin{align*}
A_W/\mathcal{I}_W&\cong
C^*\big(\mathcal{U}^m \widehat M_F\mid m\in \Z_{\ge0},\, F\in c(\Z_{\ge0}, C(\Z_s))\big)\\
&=C^*\big(\mathcal{U}^m\nu_{c}\mu_f\mid m\in \Z_{\ge0},\, c\in c(\Z_{\ge0}),\, f\in C(\Z_s)\big)\cong C(\Z_s)\otimes \mathcal{T}\bigl(\ell^2(\Z_{\ge0})\bigr).\tag*{\qed}
\end{align*}\renewcommand{\qed}{}
\end{proof}

The main result of this section in the theorem below is then a direct consequence of equation~\eqref{SESequence}, Propositions~\ref{IWStructure} and~\ref{WAlgebraQuotient}.
\begin{Theorem}
We have the short exact sequence of $C^*$-algebras
\begin{equation*}
0\rightarrow \mathcal{K}\rightarrow A_W\rightarrow \mathcal{T}\otimes C(\Z_s)\rightarrow 0 .
\end{equation*}
\end{Theorem}

\subsection[K-theory of A\_W]{$\boldsymbol{K}$-theory of $\boldsymbol{A_W}$}

The above sequence, paired with the K\"unneth formula for $K$-theory, allows us to compute the $K$-theory of $A_W$.
\begin{Proposition}
The $K$-theory of is given by
$K_0(A_W) \cong C(\Z_s,\Z) \oplus \Z$ and $ K_1(A_W) \cong 0$.
\end{Proposition}
\begin{proof}
We begin by using the K\"unneth formula to compute the $K$-theory of $\mathcal{T} \otimes C(\Z_s)$. Since the $K$-theory of $\mathcal{T}$ is free, by \cite{SC}, we have the following isomorphisms:
\begin{gather*}
K_0(\mathcal{T} \otimes C(\Z_s)) \cong (K_0(\mathcal{T}) \otimes K_0(C(\Z_s))) \oplus ( K_1(\mathcal{T}) \otimes K_1(C(\Z_s))),
\\
K_1(\mathcal{T} \otimes C(\Z_s)) \cong (K_0(\mathcal{T}) \otimes K_1(C(\Z_s))) \oplus ( K_1(\mathcal{T}) \otimes K_0(C(\Z_s))).
\end{gather*}
The $K$-theory of $\mathcal{T}$ is well known: $K_0(\mathcal{T}) \cong \Z$, and $K_1(\mathcal{T}) \cong 0$. Since $\Z_s$ is totally disconnected, $C(\Z_s)$ is AF by \cite[Example III.2.5]{KD}. Moreover, by \cite[Exercise 3.4]{RLL} its $K$-theory can be computed as $K_0(C(\Z_s)) \cong C(\Z_s,\Z)$, and $K_1(C(\Z_s)) \cong 0$. Plugging this into the above shows that
$
K_0(\mathcal{T} \otimes C(\Z_s)) \cong C(\Z_s,\Z)$ and $ K_1(\mathcal{T} \otimes C(\Z_s)) \cong 0$.
Using this, we can compute $K_*(A_W)$ from the $6$-term exact sequence in $K$-theory,
\begin{equation*}
\begin{tikzcd}
 K_0(\mathcal{K}) \arrow{r} & K_0(A_W) \arrow{r} & K_0(\mathcal{T} \otimes C(\Z_s)) \arrow{d}{\textup{exp}} \\
 K_1(\mathcal{T} \otimes C(\Z_s)) \arrow{u}{\textup{ind}} & K_1(A_W) \arrow{l} & K_1(\mathcal{K}) \arrow{l}.
\end{tikzcd}
\end{equation*}
Since both $K_1(\mathcal{K})$ and $K_1(\mathcal{T} \otimes C(\Z_s))$ are zero, it follows immediately by exactness on the bottom row that
$K_1(A_W) \cong 0$.
To compute $K_0(A_W)$, we extract a short exact sequence from the top row
$0 \to \Z \to K_0(A_W) \to C(\Z_s,\Z) \to 0$.
Since $C(\Z_s,\Z)$ is free, by \cite[Exercise~7.7.5]{HR}, we see immediately that this sequence right splits. Hence, by the splitting lemma,
$K_0(A_W) \cong \Z \oplus C(\Z_s,\Z)$.
This completes the proof.
\end{proof}

\section{Structure summary}
For convenience, the following table provides a short summary of relevant ideals and the corresponding quotients for the four shift algebras discussed in the above sections.

\begin{center}
\begin{tabular}{c||c|c|c|c|c}\renewcommand{\arraystretch}{1.3}
Index & Coefficient algebra & Ideal & Quotient & Promotion & Quotient \\
$\mathcal{J}$ & $B_\mathcal{J}$& $I_\mathcal{J}$ & $B_\mathcal{J}/I_\mathcal{J}$ & $\mathcal{I}_\mathcal{J}$ & $A_\mathcal{J}/\mathcal{I}_\mathcal{J}$ \\
\hline
\hline
$U$ & $C^*(M_f,P_n)$ & $C^*(M_fP_n)$ & $C(\Z_s)$ & $\kcal \otimes C(\Z_s)$ & $BD(s^
\infty)$ \\
& $\cong c(\Z_{\ge 0},C(\Z_s))$ & $\cong c_0(\Z_{\ge 0},C(\Z_s))$ & & & \\
\hline
$V$ & $C^*(M_f,P_{(n,0)})$ & $C^*(P_{(n,0)})$ & $C(\Z_s)$ & $\kcal$ & ${\rm HS}(s)$ \\
& $\cong C(X_V)$ & $\cong c_0(\Z_{\ge 0})$ & & & \\
\hline
$S$ & $C^*(S_{(n,x),(n,y)})$ & $\kcal(H)_{\mathrm{inv}}$ & $\mathcal{O}_{s,\mathrm{inv}}$ & $\kcal(H)$ & $\mathcal{O}_s$ \\
\hline
$W$ & $C^*(M_F,\pcal_n)$ & $\kcal(H)_{\mathrm{inv}}$ & $c(\Z_{\ge 0})\otimes C(\Z_s)$ & $\kcal(H)$ & $\mathcal{T}\otimes C(\Z_s)$
\end{tabular}
\end{center}

\subsection*{Acknowledgements}

The authors wish to thank the referees for their very insightful comments and suggestions. It should be noted that one of the referees pointed out that some of our results can also be obtained using the theory of the Toeplitz $C^*$-algebra of a $C^*$-correspondence as developed by \cite{Pim}. While this gives a very intriguing alternative, possibly simpler technique of computing the $K$-theory, we decided not to follow this approach since the theory of $C^*$-correspondence is separate from our unified approach based on coefficient algebras and crossed products by endomorphisms.

\pdfbookmark[1]{References}{ref}
\LastPageEnding

\end{document}